\documentclass[12pt]{article}
\usepackage{amssymb}
\usepackage{amsmath}
\usepackage{amsfonts}

\newtheorem{theorem}{Theorem}

\newenvironment{proof}[1][Proof]{\noindent\textbf{#1.} }{\ \rule{0.5em}{0.5em}}

\begin{document}

\title{A simplified proof of CLT for convex bodies}
\author{Daniel J. Fresen\thanks{%
University of Pretoria, Department of Mathematics and Applied Mathematics,
daniel.fresen@up.ac.za or djfb6b@mail.missouri.edu. MSC2010: 52A20, 52A23, 60F05 (Primary), 26B25, 52A38, 62E20 (Secondary). Keywords: central limit theorem for convex bodies, log-concave function.}}
\maketitle

\begin{abstract}
We present a short proof of Klartag's central limit theorem for convex
bodies, using only the most classical facts about log-concave functions. An
appendix is included where we give the proof that thin shell implies CLT.
The paper is accessible to anyone.
\end{abstract}

\section{Introduction}

The central limit theorem for convex bodies (Theorem \ref{CLT conv bodies}
below) was conjectured by Brehm and Voigt \cite{BrVo} and independently (at
about the same time) by Anttila, Ball and Perissinaki \cite{ABP}. A 1998
preprint of \cite{ABP} is cited in \cite{Bo}. It took several years and
various partial results before a full proof by Klartag emerged in \cite{Kl}
(see p95 for the history). A different proof was given soon afterwards by
Fleury, Gu\'{e}don, and Paouris \cite{FGP}. Significantly improved
quantitative bounds (from logarithmic to power type) were given by Klartag 
\cite{Kl4}, followed by improved estimates by various authors on the related 'thin shell
property' \cite{Fl, GuMi11, LeVe17}. More information can be found in \cite%
{Fl, GuMi11, Kl, Kl4, Kl2, LeVe18}.

We present a simple proof that is self-contained (except for very classical
results such as the Pr\'{e}kopa-Leindler inequality) and is accessible to
anyone. The bounds on $\varepsilon _{n}$ and $\omega _{n}$ that this proof
gives are poor; the contribution is simplicity. The methodology is a
variation of that in Klartag's original proof and uses Fourier inversion;
the main difference being that we apply concentration directly to the
Fourier transform as opposed to the measure of half-spaces. The statement of
Theorem \ref{CLT conv bodies} below is not identical to Theorem 1.1 in \cite%
{Kl}, however under log-concavity, a uniform estimate on the cumulative
distribution gives an estimate on the total variation distance, so we do
indeed recover Theorem 1.1 in \cite{Kl}. The standard Euclidean norm and
inner product on $\mathbb{R}^{n}$ are denoted as $\left\vert \cdot
\right\vert $ and $\left\langle \cdot ,\cdot \right\rangle $ respectively.

\begin{theorem}
\label{CLT conv bodies}There exist sequences $\left( \varepsilon _{n}\right)
_{1}^{\infty }$ and $\left( \omega _{n}\right) _{1}^{\infty }$ in $\left(
0,\infty \right) $ with $\lim_{n\rightarrow \infty }\varepsilon
_{n}=\lim_{n\rightarrow \infty }\omega _{n}=0$ such that the following is
true: Let $n\in \mathbb{N}$, let $X$ be a random vector in $\mathbb{R}^{n}$
with $\mathbb{E}X=0$ and $\mathrm{Cov}\left( X\right) =I_{n}$. Assume that $%
X $ has a density $f=d\mu /dx$ that is log-concave, i.e. $f=e^{-g}$ where $g:%
\mathbb{R}^{n}\rightarrow \left( -\infty ,\infty \right] $ is convex. Then
there exists a set $\Theta \subset S^{n-1}$ with $\sigma _{n-1}\left(
S^{n-1}\right) \geq 1-\omega _{n}$ such that for all $\theta \in \Theta $,%
\[
\sup_{t\in \mathbb{R}}\left\vert \mathbb{P}\left\{ \left\langle X,\theta
\right\rangle \leq t\right\} -\Phi \left( t\right) \right\vert \leq
\varepsilon _{n} 
\]%
where $\sigma _{n-1}$ is Haar measure on $S^{n-1}$ normalized so that $%
\sigma _{n-1}\left( S^{n-1}\right) =1$, and $\Phi \left( t\right) =\left(
2\pi \right) ^{-1/2}\int_{-\infty }^{t}\exp \left( -u^{2}/2\right) du$.
\end{theorem}

The proof uses two nontrivial properties of log-concave functions (see \cite%
{Kl, Kl4, Kl2} for more details): with $f$ as in Theorem \ref{CLT conv
bodies},

\noindent $\bullet $ If $E\subset \mathbb{R}^{n}$ is any linear subspace of
dimension $1\leq k<n$, then the projection $P_{E}f:E\rightarrow \left[
0,\infty \right) $ defined by%
\begin{equation}
P_{E}f(x)=\int_{E^{\bot }}f\left( x+y\right) dy  \label{proj def}
\end{equation}%
is log-concave. Here integration is performed with respect to $n-k$
dimensional Lebesgue measure on $E^{\bot }$. This is a consequence of the Pr%
\'{e}kopa-Leindler inequality. Interpreting a convolution in terms of a
projection of $\mathbb{R}^{n}\times \mathbb{R}^{n}$ onto $\mathbb{R}^{n}$,
we see that if $\varphi :\mathbb{R}^{n}\rightarrow \left[ 0,\infty \right) $
is log-concave with $\int_{\mathbb{R}^{n}}\varphi (x)dx=1$, then the
convolution $f\ast \varphi $ is also log-concave.

\noindent $\bullet $ If $X$ has the thin shell property, i.e.%
\[
\mathbb{P}\left\{ \left\vert \frac{\left\vert X\right\vert }{R}-1\right\vert
<\varepsilon ^{\prime }\right\} >1-\varepsilon ^{\prime } 
\]%
for some $\varepsilon ^{\prime },R>0$ (here we can take $R=\sqrt{n}$), then
the projection of $X$ onto most one dimensional subspaces is approximately
Gaussian, with estimates depending on $\varepsilon ^{\prime }$. Quantitative
results of this type for log-concave measures can be found in \cite{ABP, Bo}%
. For completeness, we give a precise statement with proof in Section \ref%
{appendix thin shell implies CLT}.

\section{Proof of Theorem \protect\ref{CLT conv bodies}}

The proof is in three main steps.

\textbf{Step 1: Approximately spherically symmetric projections.} The first
step mimics Milman's proof of Dvoretzky's theorem \cite{Milman}, see for
example \cite{Sch2}, but in a different way to Klartag \cite[Sections 3 and 4%
]{Kl}. Let $Y=X+\sigma Z$ for some $\sigma >0$, where $Z$ has the standard
normal distribution and is independent of $X$. The density of $Y$ is $%
h=f\ast \phi _{\sigma }$, where $\phi _{\sigma }(x)=\left( 2\pi \sigma
^{2}\right) ^{-n/2}\exp \left( -2^{-1}\sigma ^{-2}\left\vert x\right\vert
^{2}\right) $ and $\ast $ denotes convolution. Then $\widehat{h}=\widehat{f}%
\cdot \widehat{\phi }_{\gamma }$, where $\widehat{\cdot }$ denotes the
Fourier transform,%
\[
\widehat{h}\left( \xi \right) =\int_{\mathbb{R}^{n}}\exp \left( -2\pi
i\left\langle \xi ,x\right\rangle \right) h\left( x\right) dx 
\]%
and%
\[
\widehat{\phi }_{\sigma }(\xi )=\exp \left( -2\pi ^{2}\sigma ^{2}\left\vert
\xi \right\vert ^{2}\right) 
\]%
For any $\xi _{1},\xi _{2}\in \mathbb{R}^{n}$,%
\begin{eqnarray*}
\left\vert \widehat{f}\left( \xi _{1}\right) -\widehat{f}\left( \xi
_{2}\right) \right\vert &\leq &\int_{\mathbb{R}^{n}}\left\vert \exp \left(
-2\pi i\left\langle \xi _{1},x\right\rangle \right) -\exp \left( -2\pi
i\left\langle \xi _{2},x\right\rangle \right) \right\vert f(x)dx \\
&\leq &\int_{\mathbb{R}^{n}}2\pi \left\vert \left\langle \xi
_{1},x\right\rangle -\left\langle \xi _{2},x\right\rangle \right\vert f(x)dx
\\
&=&2\pi \left\vert \xi _{1}-\xi _{2}\right\vert \int_{\mathbb{R}%
^{n}}\left\vert \left\langle \frac{\xi _{1}-\xi _{2}}{\left\vert \xi
_{1}-\xi _{2}\right\vert },x\right\rangle \right\vert f(x)dx \\
&\leq &2\pi \left\vert \xi _{1}-\xi _{2}\right\vert \left( \mathbb{E}%
\left\vert \left\langle \frac{\xi _{1}-\xi _{2}}{\left\vert \xi _{1}-\xi
_{2}\right\vert },X\right\rangle \right\vert ^{2}\right) ^{1/2}
\end{eqnarray*}%
and we see that $\widehat{f}$ is $2\pi $-Lipschitz on $\mathbb{R}^{n}$. Let $%
F\in G_{n,k}$ be any fixed subspace and $U$ a random matrix uniformly
distributed in $O(n)$ ($k<n$ to be determined later). Then $E=UF\in G_{n,k}$
is a random $k$-dimensional subspace uniformly distributed in $G_{n,k}$. Let 
$\varepsilon \in \left( 0,1/2\right) $ and let $\mathcal{N}\subset
S_{F}=S^{n-1}\cap F$ be an $\varepsilon $-dense subset (i.e. for all $\theta
\in S_{F}$ there exists $\omega \in \mathcal{N}$ such that $\left\vert
\theta -\omega \right\vert <\varepsilon $. By considering the volume of
disjoint balls, such a subset can be chosen with cardinality $\left\vert 
\mathcal{N}\right\vert \leq \left( 3/\varepsilon \right) ^{k}$. Assume that $%
k\leq c\left( \log \varepsilon ^{-1}\right) ^{-1}\delta n$. By L\'{e}vy's
concentration inequality for Lipschitz functions on a sphere, see e.g. \cite%
{Kl2}, and the union bound, with probability at least%
\[
1-\sum_{m=0}^{\infty }\left( \frac{3}{\varepsilon }\right) ^{k}\exp \left(
-\left\{ \sqrt{c\delta ^{2}+\frac{2\ln m}{n}}\right\} ^{2}n\right) \geq
1-C\exp \left( -c\delta ^{2}n\right) 
\]%
the following event occurs: for all $m\in \left\{ 0,1,2\ldots \right\} $,
and all $\theta \in \mathcal{N}$,%
\[
\left\vert \widehat{f}\left( U\left( 1+\varepsilon \right) ^{m}\sqrt{k}%
\sigma ^{-1}\theta \right) -M\left( \left( 1+\varepsilon \right) ^{m}\sqrt{k}%
\sigma ^{-1}\right) \right\vert <C\left( \delta +\sqrt{\frac{\ln m}{n}}%
\right) \left( 1+\varepsilon \right) ^{m}\sigma ^{-1}\sqrt{k} 
\]%
where%
\[
M\left( t\right) =\int_{S^{n-1}}\widehat{f}\left( t\theta \right) d\sigma
_{n-1}\left( \theta \right) 
\]%
With the same probability, the same event holds with $\left( 1+\varepsilon
\right) ^{m}$ replaced with $\left( 1+\varepsilon \right) ^{-m}$. Setting $%
\xi ^{\prime }=\left( 1+\varepsilon \right) ^{\pm m}\sqrt{k}\sigma
^{-1}\theta $, making $m$ the subject of the formula, and using the
Lipschitz property of $\widehat{f}$, with high probability, for all $\xi \in
F$,%
\[
\left\vert \widehat{f}\left( U\xi \right) -M\left( \left\vert \xi
\right\vert \right) \right\vert <C\left( \delta +\varepsilon +\sqrt{\frac{%
\ln \varepsilon ^{-1}}{n}}+\sqrt{\frac{1}{n}\ln \ln \max \left\{ \frac{%
\sigma \left\vert \xi \right\vert }{\sqrt{k}},\frac{\sqrt{k}}{\sigma
\left\vert \xi \right\vert }\right\} }\right) \left\vert \xi \right\vert 
\]%
Optimizing over $\varepsilon $ we set $\varepsilon =\sqrt{\left( \ln
n\right) /n}$. Let $P_{E}:\mathbb{R}^{n}\rightarrow E$ denote the orthogonal
projection onto $E$, let $\mathcal{F}_{\mathbb{R}^{n}}:$ $L^{1}\left( 
\mathbb{R}^{n}\right) \rightarrow L^{\infty }\left( \mathbb{R}^{n}\right) $
denote the Fourier transform on $\mathbb{R}^{n}$ and let $\mathcal{F}%
_{E}:L^{1}\left( E\right) \rightarrow L^{\infty }\left( E\right) $ denote
the Fourier transform on $E$ ($E$ as a Hilbert space in its own right).
Recall the definition in (\ref{proj def}). By Fubini's theorem, the function 
$P_{E}h$ is the density of the random vector $P_{E}X$ (with respect to $k$%
-dimensional Lebesgue measure in $E$). The Fourier transform works well with
orthogonal projections, in particular%
\[
\left( \mathcal{F}_{\mathbb{R}^{n}}h\right) |_{E}=\mathcal{F}_{E}\left(
P_{E}h\right) 
\]%
where $\left( \mathcal{F}_{\mathbb{R}^{n}}f\right) |_{E}$ denotes the
restriction of $\mathcal{F}_{\mathbb{R}^{n}}f$ to $E$. By Fourier inversion
in $E$, for all $x\in E$,%
\[
P_{E}h\left( x\right) =\int_{E}\exp \left( 2\pi i\left\langle x,\xi
\right\rangle \right) \widehat{h}\left( \xi \right) d\xi 
\]%
so for all $W\in O\left( E\right) $, (applying a change of variables)%
\begin{eqnarray}
&&\left\vert P_{E}h\left( x\right) -P_{E}h\left( Wx\right) \right\vert 
\nonumber \\
&\leq &\int_{E}\left\vert \widehat{h}\left( \xi \right) -\widehat{h}\left(
W\xi \right) \right\vert d\xi  \nonumber \\
&\leq &C\left( 2\pi \sigma ^{2}\right) ^{-(k+1)/2}\int_{E}\left( \delta +%
\sqrt{\frac{\ln n}{n}}+\sqrt{\frac{1}{n}\ln \ln \max \left\{ \frac{%
\left\vert y\right\vert }{\sqrt{2\pi k}},\frac{\sqrt{2\pi k}}{\left\vert
y\right\vert }\right\} }\right) e^{-\pi \left\vert y\right\vert
^{2}}\left\vert y\right\vert dy  \nonumber \\
&\leq &C\left( 2\pi \sigma ^{2}\right) ^{-(k+1)/2}\left( \delta +\sqrt{\frac{%
\ln n}{n}}\right) \sqrt{k}  \label{grape}
\end{eqnarray}

\textbf{Step 2: Behavior of }$t\mapsto P_{E}h\left( t\theta \right) $\textbf{%
\ }(in the spirit of Lemmas 4.3 and 4.4 in \cite{Kl}). Consider any $x,y\in
S_{E}=E\cap S^{n-1}$ and define $A,B:\left[ 0,\infty \right) \rightarrow 
\mathbb{R}$ by%
\[
P_{E}h\left( tx\right) =e^{-A(t)}\hspace{0.45in}P_{E}h\left( ty\right)
=e^{-B(t)} 
\]%
Since $f$ and $\phi $ are log-concave, i.e. $-\log f$ and $-\log \phi $ are
convex with values in $\left( -\infty ,\infty \right] $, $h=f\ast \phi $ is
also log-concave. It follows from the Pr\'{e}kopa-Leindler inequality (see
for example the discussion in \cite{Kl}) that $P_{E}h$ too is log-concave,
and therefore $A$ and $B$ are convex. Since $P_{E}h=\left( P_{E}f\right)
\ast \left( P_{E}\phi _{\sigma }\right) $, $A$ and $B$ are infinitely
differentiable. In preparation for an integral over $E$ in polar
coordinates, we now study $t\mapsto t^{k-1}e^{-A(t)}$ and $t\mapsto
t^{k-1}e^{-B(t)}$, $t\in \left[ 0,\infty \right) $. These functions are
maximized at $t_{x}$,$t_{y}\in \left( 0,\infty \right) $ that satisfy%
\[
A^{\prime }\left( t_{x}\right) t_{x}=k-1\hspace{0.45in}B^{\prime }\left(
t_{y}\right) t_{y}=k-1 
\]%
Such numbers exist since $A^{\prime }(t)t$ is continuous with limit $0$
(resp. $\infty $) as $t\rightarrow 0$ (resp. $t\rightarrow \infty $),
similarly for $B$. After a possible re-labeling of $x$ and $y$ we may assume
that $t_{x}\leq t_{y}$. Our goal is to show that these numbers cannot be too
far apart (in the sense that their ratio is close to $1$). If $t_{x}=t_{y}$
there is nothing to show, so assume $t_{x}<t_{y}$. By convexity,%
\begin{eqnarray*}
A\left( t_{y}\right) -A\left( t_{x}\right) &\geq &A^{\prime }\left(
t_{x}\right) \left( t_{y}-t_{x}\right) =\left( k-1\right) \left( \frac{t_{y}%
}{t_{x}}-1\right) \\
B\left( t_{y}\right) -B\left( t_{x}\right) &\leq &B^{\prime }\left(
t_{y}\right) \left( t_{y}-t_{x}\right) =\left( k-1\right) \left( 1-\frac{%
t_{x}}{t_{y}}\right)
\end{eqnarray*}%
and therefore%
\begin{equation}
\sup_{t\in \left\{ t_{x},t_{y}\right\} }\left\vert A(t)-B(t)\right\vert \geq 
\frac{\left\{ A\left( t_{y}\right) -A\left( t_{x}\right) \right\} -\left\{
B\left( t_{y}\right) -B\left( t_{x}\right) \right\} }{2}=\frac{\left(
k-1\right) \left( t_{y}-t_{x}\right) ^{2}}{2t_{x}t_{y}}  \label{no name aae}
\end{equation}%
Assume momentarily that there exists $t\in \left\{ t_{x},t_{y}\right\} $
such that $A(t)-B(t)\geq 1$. Since $P_{E}h$ is the log-concave density of a
random vector in $E$ with covariance $\left( 1+\sigma ^{2}\right) I$, it
follows from Theorem 5.14 in \cite{LoVe} (see also (\ref{lower bound at zero}%
) here) that $P_{E}h\left( 0\right) \geq 2^{-7k}\left( 1+\sigma ^{2}\right)
^{-k/2}$. By convexity again, 
\begin{eqnarray*}
\left\vert e^{-A(t)}-e^{-B(t)}\right\vert &=&e^{-B(t)}\left\vert
e^{B(t)-A(t)}-1\right\vert \geq \left( 1-e^{-1}\right) e^{-B(t)} \\
&\geq &\left( 1-e^{-1}\right) \exp \left( -B(0)-tB^{\prime }\left( t\right)
\right) \\
&\geq &\left( 1-e^{-1}\right) P_{E}h\left( 0\right) \exp \left(
-t_{y}B^{\prime }\left( t_{y}\right) \right) \\
&\geq &\left( e-1\right) 2^{-7k}\left( 1+\sigma ^{2}\right) ^{-k/2}\exp
\left( -k\right)
\end{eqnarray*}%
However, by (\ref{grape}),%
\[
\left\vert e^{-A(t)}-e^{-B(t)}\right\vert =\left\vert P_{E}h\left( tx\right)
-P_{E}h\left( ty\right) \right\vert \leq C\left( 2\pi \sigma ^{2}\right)
^{-(k+1)/2}\left( \delta +\sqrt{\frac{\ln n}{n}}\right) \sqrt{k} 
\]%
We will choose the parameters $\delta $, $k$, and $\sigma $ so that the
upper bound on $\left\vert e^{-A(t)}-e^{-B(t)}\right\vert $ is less than the
lower bound, which implies that we may assume that $B(t)-A(t)>-1$ for all $%
t\in \left\{ t_{x},t_{y}\right\} $. Now let $t\in \left\{
t_{x},t_{y}\right\} $ such that%
\[
\left\vert A(t)-B(t)\right\vert =\sup_{u\in \left\{ t_{x},t_{y}\right\}
}\left\vert A(u)-B(u)\right\vert 
\]%
By (\ref{no name aae}),%
\begin{eqnarray*}
\left\vert e^{-A(t)}-e^{-B(t)}\right\vert &=&e^{-B(t)}\left\vert
e^{B(t)-A(t)}-1\right\vert \\
&\geq &\exp \left( -B(0)-B^{\prime }(t)t\right) e^{-1}\left\vert
B(t)-A(t)\right\vert \\
&\geq &2^{-7k}\left( 1+\sigma ^{2}\right) ^{-k/2}e^{-k}\frac{\left(
k-1\right) \left( t_{y}-t_{x}\right) ^{2}}{2t_{x}t_{y}}
\end{eqnarray*}%
so%
\[
\frac{t_{y}-t_{x}}{t_{y}}\leq \gamma :=Ce^{ck}\left( 1+\sigma ^{2}\right)
^{k/4}\sigma ^{-(k+1)/2}\left( \delta ^{1/2}+\left( \frac{\ln n}{n}\right)
^{1/4}\right) 
\]%
For an appropriate choice of parameters this will achieve our goal of
showing that $t_{x}$ and $t_{y}$ cannot be too far apart (relatively). What
this means is that in any direction $x\in S^{n-1}\cap E$, the function $%
t\mapsto t^{k-1}P_{E}h\left( tx\right) $ achieves its peak in about the same
place. Our next goal is to show that the mass in%
\[
\int_{0}^{\infty }t^{k-1}P_{E}h\left( tx\right) dt 
\]%
is concentrated around $t_{x}$. Since $A$ lies above its tangent lines,
defining $q$ by%
\begin{eqnarray*}
q(t) &=&t^{k-1}e^{-A(t)}\leq \exp \left( \left( k-1\right) \ln t-A\left(
t_{x}\right) -\left( t-t_{x}\right) A^{\prime }\left( t_{x}\right) \right) \\
&=&\exp \left( -A\left( t_{x}\right) -\left( \frac{t}{t_{x}}-1-\ln \frac{t}{%
t_{x}}-\ln t_{x}\right) \left( k-1\right) \right) \\
&=&\exp \left( -A\left( t_{x}\right) -\left( -\ln t_{x}+\sum_{j=2}^{\infty
}j^{-1}\left( \frac{t}{t_{x}}-1\right) ^{j}\right) \left( k-1\right) \right)
\\
&\leq &t_{x}^{k-1}e^{-A\left( t_{x}\right) }\exp \left( -\frac{k-1}{3}\left( 
\frac{t}{t_{x}}-1\right) ^{2}\right)
\end{eqnarray*}%
provided $\left\vert \frac{t}{t_{x}}-1\right\vert <1/2$. We now translate
this to tail probabilities. Fix any $t\in \left[ t_{x},3t_{x}/2\right] $ and 
$s\geq t$. By log-concavity of $q$,%
\[
q\left( s\right) \leq \left[ \left( \frac{q(t)}{q\left( t_{x}\right) }%
\right) ^{1/\left( t-t_{x}\right) }\right] ^{s-t}q(t)\leq \exp \left( -\frac{%
\left( k-1\right) \left( s-t\right) \left( t-t_{x}\right) }{3t_{x}^{2}}%
\right) q(t) 
\]%
and therefore%
\[
\int_{t}^{\infty }q(s)ds\leq \frac{3t_{x}^{2}q(t)}{\left( k-1\right) \left(
t-t_{x}\right) } 
\]%
On the other hand, for any $s\in \left[ t_{x},t\right] $,%
\[
q\left( s\right) \geq \left[ \left( \frac{q(t_{x})}{q\left( t\right) }%
\right) ^{1/\left( t-t_{x}\right) }\right] ^{t-s}q(t)\geq \exp \left( \frac{%
\left( k-1\right) \left( t-s\right) \left( t-t_{x}\right) }{3t_{x}^{2}}%
\right) q(t) 
\]%
so%
\[
\int_{0}^{\infty }q(s)ds\geq \int_{t_{x}}^{t}q(s)ds\geq \frac{3t_{x}^{2}q(t)%
}{\left( k-1\right) \left( t-t_{x}\right) }\left[ \exp \left( \frac{\left(
k-1\right) \left( t-t_{x}\right) ^{2}}{3t_{x}^{2}}\right) -1\right] 
\]%
and%
\[
\int_{t}^{\infty }q(s)ds\leq \left[ \exp \left( \frac{\left( k-1\right)
\left( t-t_{x}\right) ^{2}}{3t_{x}^{2}}\right) -1\right] ^{-1}\int_{0}^{%
\infty }q(s)ds 
\]%
A similar bound holds for the left hand tail. Combining these,%
\begin{equation}
\int_{(1-u)t_{x}}^{(1+u)t_{x}}q(s)ds\geq \left( 1-C\exp \left(
-cku^{2}\right) \right) \left( \int_{0}^{\infty }q(s)ds\right)
\label{radial conc}
\end{equation}%
provided $u\in \left[ 0,1/2\right] $.

\textbf{Step 3: Thin shell and small details.} Now fix an arbitrary $x\in
B_{2}^{n}\cap E$. By polar integration,%
\begin{equation}
\mathbb{P}\left\{ \left\vert \frac{\left\vert P_{E}Y\right\vert }{t_{x}}%
-1\right\vert <C\left( u+\gamma \right) \right\} \geq 1-C\exp \left(
-cku^{2}\right)  \label{first thin shell}
\end{equation}%
which is the so called 'thin shell property' of $P_{E}Y$ in $E$ (see Section %
\ref{appendix thin shell implies CLT} for more details), and by a result of
Bobkov \cite{Bo} (following Anttila, Ball and Perissinaki \cite{ABP} in the
symmetric case) this implies that with probability at least%
\[
1-C\sqrt{k}\exp \left( -ck\left\{ u+\gamma +\exp \left( -cku^{2}\right)
\right\} ^{2}\right) 
\]%
a further random projection $P_{\theta ^{\prime }}P_{E}Y$ is approximately
Gaussian (with mean zero and variance $1+\sigma ^{2}$), where $\theta
^{\prime }$ is uniformly distributed in $S_{E}$,%
\[
\left\vert \mathbb{P}\left\{ \left\langle \theta ^{\prime
},P_{E}Y\right\rangle \leq t\right\} -\Phi \left( \frac{t}{\sqrt{1+\sigma
^{2}}}\right) \right\vert \leq C\left( u+\gamma +\exp \left( -cku^{2}\right)
\right) 
\]%
See Theorem \ref{thin shell imp CLT}. Now $\left\langle \theta ^{\prime
},P_{E}Y\right\rangle =\left\langle \theta ^{\prime },P_{E}X\right\rangle
+\left\langle \theta ^{\prime },\sigma P_{E}Z\right\rangle $, and $%
\left\langle \theta ^{\prime },P_{E}Z\right\rangle \sim N(0,1)$. Assume that 
$t\geq 0$ and $\sigma \leq 1$, and consider any $\nu \in \left( 0,1\right) $%
. Since%
\begin{eqnarray*}
\left\{ \left\langle \theta ^{\prime },P_{E}Y\right\rangle \leq t-\nu
\right\} &\Rightarrow &\left\{ \left\langle \theta ^{\prime
},P_{E}X\right\rangle \leq t\right\} \vee \left\{ \left\langle \theta
^{\prime },\sigma P_{E}Z\right\rangle \leq -\nu \right\} \\
\left\{ \left\langle \theta ^{\prime },P_{E}X\right\rangle \leq t\right\}
&\Rightarrow &\left\{ \left\langle \theta ^{\prime },P_{E}Y\right\rangle
\leq t+\nu \right\} \vee \left\{ \left\langle \theta ^{\prime },\sigma
P_{E}Z\right\rangle \geq \nu \right\}
\end{eqnarray*}%
by the union bound and (\ref{log lipschitz}), $\mathbb{P}\left\{
\left\langle \theta ^{\prime },P_{E}X\right\rangle \leq t\right\} $ is
bounded below by%
\begin{eqnarray*}
&&\mathbb{P}\left\{ \left\langle \theta ^{\prime },P_{E}Y\right\rangle \leq
t-\nu \right\} -\mathbb{P}\left\{ \left\langle \theta ^{\prime },\sigma
P_{E}Z\right\rangle \leq -\nu \right\} \\
&\geq &\Phi \left( \frac{t-\nu }{\sqrt{1+\sigma ^{2}}}\right) -C\left(
u+\gamma +\exp \left( -cku^{2}\right) \right) -C\exp \left( -c\sigma
^{-2}\nu ^{2}\right) \\
&\geq &\Phi \left( t\right) -C\left( \nu +\sigma +u+\gamma +\exp \left(
-cku^{2}\right) +\exp \left( -c\sigma ^{-2}\nu ^{2}\right) \right)
\end{eqnarray*}%
and above by%
\begin{eqnarray*}
&&\mathbb{P}\left\{ \left\langle \theta ^{\prime },P_{E}Y\right\rangle \leq
t+\nu \right\} +\mathbb{P}\left\{ \left\langle \theta ^{\prime },\sigma
P_{E}Z\right\rangle \geq \nu \right\} \\
&\leq &\Phi \left( t\right) +C\left( \nu +\sigma +u+\gamma +\exp \left(
-cku^{2}\right) +\exp \left( -c\sigma ^{-2}\nu ^{2}\right) \right)
\end{eqnarray*}%
Choosing%
\begin{eqnarray*}
k &=&\frac{c_{1}\ln \left( n+1\right) }{\ln \ln \left( n+2\right) }\hspace{%
0.65in}\delta =\frac{\ln \left( n+1\right) }{\sqrt{n}}\hspace{0.65in}\sigma =%
\frac{1}{\ln \left( n+1\right) } \\
u &=&\frac{C_{2}\ln \ln \left( n+2\right) }{\sqrt{\ln \left( n+1\right) }}%
\hspace{0.45in}\nu =\frac{C_{2}}{\sqrt{\ln \left( n+1\right) }}
\end{eqnarray*}%
(a fairly arbitrary choice), where $c_{1}$ is chosen first to be small and
then $C_{2}$ is chosen to be appropriately large, we get $\gamma \leq
Cn^{-1/5}$ and the error bound reduces to%
\[
\left\vert \mathbb{P}\left\{ \left\langle \theta ^{\prime
},P_{E}X\right\rangle \leq t\right\} -\Phi (t)\right\vert \leq \delta _{n}:=%
\frac{C\ln \ln \left( n+2\right) }{\sqrt{\ln \left( n+1\right) }} 
\]%
the probability bound (of failure) reduces to%
\[
\omega _{n}\leq C\exp \left( -c\delta ^{2}n\right) +C\sqrt{k}\exp \left(
-ck\left\{ u+\gamma +\exp \left( -cku^{2}\right) \right\} ^{2}\right) \leq
C\left( \log n\right) ^{-C_{3}} 
\]%
where $C_{3}$ can be made arbitrarily large by taking $C_{2}$ large enough.
The upper and lower bounds for $\left\vert e^{-A(t)}-e^{-B(t)}\right\vert $
earlier in the proof become (respectively) $Cn^{-1/2+0.1}$ and $Cn^{-0.1}$,
which achieves the desired contradiction, and the required bound $k\leq
c\delta ^{2}\left( \ln n\right) ^{-1}n$ is satisfied. Note that $P_{\theta
^{\prime }}P_{E}=P_{\theta }$ where $\theta $ is uniformly distributed in $%
S^{n-1}$, so we have shown that the projection of $X$ onto most one
dimensional subspaces is approximately Gaussian, and Theorem \ref{CLT conv
bodies} follows.

\textbf{Note: Radius of the thin shell.} When stating and applying the fact
that the thin shell property implies CLT, it is convenient to replace $t_{x}$
with $\sqrt{k}$ in (\ref{first thin shell}). Let $W_{\theta }$ ($\theta \in
S^{n-1}\cap E$) be a random variable with density proportional to $q_{\theta
}(t)=t^{k-1}P_{E}h\left( t\theta \right) $, $t\geq 0$. From (\ref{radial
conc}),%
\begin{eqnarray*}
\mathbb{E}\left\vert W_{\theta }\right\vert ^{2} &=&\left( \mathbb{E}%
\left\vert W_{\theta }\right\vert \right) ^{2}+\mathrm{Var}\left( W_{\theta
}\right) \leq \left( t_{\theta }+Ck^{-1/2}t_{\theta }\right) ^{2}+\frac{%
Ct_{\theta }^{2}}{k} \\
&\leq &t_{x}^{2}\left( 1+C\gamma \right) \left( 1+Ck^{-1/2}\right) +\frac{%
Ct_{x}^{2}}{k}
\end{eqnarray*}%
so%
\begin{eqnarray*}
\mathbb{E}\left\vert P_{E}Y\right\vert ^{2} &=&\mathrm{vol}_{k-1}\left(
S^{k-1}\right) \int_{S^{n-1}\cap E}\left( \int_{0}^{\infty }t^{2}q_{\theta
}(t)\frac{dt}{\int_{0}^{\infty }q_{\theta }(s)ds}\right) \left(
\int_{0}^{\infty }q_{\theta }(s)ds\right) d\sigma _{k-1}\left( \theta \right)
\\
&\leq &\left( 1+C\gamma +Ck^{-1/2}\right) t_{x}^{2}
\end{eqnarray*}%
The last inequality follows since $\mathrm{vol}_{k-1}\left( S^{k-1}\right)
\int_{S^{n-1}\cap E}\int_{0}^{\infty }q_{\theta }(s)dsd\sigma _{k-1}\left(
\theta \right) =1$. Similarly,%
\[
\mathbb{E}\left\vert P_{E}Y\right\vert ^{2}\geq \left( 1-C\gamma
-Ck^{-1/2}\right) t_{x}^{2} 
\]%
But $\mathbb{E}\left\vert P_{E}Y\right\vert ^{2}=k$, so%
\[
\left( 1-C\gamma -Ck^{-1/2}\right) \sqrt{k}\leq t_{x}\leq \left( 1+C\gamma
+Ck^{-1/2}\right) \sqrt{k} 
\]%
and (changing the constants involved) we may replace $t_{x}$ with $\sqrt{k}$
in (\ref{first thin shell}).

\textbf{Note: Lower bound on }$P_{E}f(0)$\textbf{.} To simplify notation we
work with the original function $f:\mathbb{R}^{n}\rightarrow \left[ 0,\infty
\right) $, but the corresponding result can then be applied to $%
P_{E}f:E\rightarrow \left[ 0,\infty \right) $ by replacing $n$ with $k$. By
log-concavity, $\left\{ x\in \mathbb{R}^{n}:f(x)>f(0)\right\} $ is convex
and there exists $\theta \in S^{n-1}$ such that $\left\langle \theta
,x\right\rangle >0$ implies $f(x)\leq f(0)$. It is an interesting exercise
to show that for any log-concave random variable in $\mathbb{R}$ with zero
mean and unit variance, such as $\left\langle \theta ,X\right\rangle $, $%
\mathbb{P}\left\{ \left\langle \theta ,X\right\rangle >0\right\} \geq \beta $
for some universal constant $\beta >0$ (actually for $\beta =e^{-1}$). Now%
\[
n=\mathbb{E}\left\vert X\right\vert ^{2}\geq A^{2}\alpha _{n}^{2}\frac{n}{%
2\pi e}\mathbb{P}\left\{ \left\vert X\right\vert \geq A\alpha _{n}\sqrt{%
\frac{n}{2\pi e}}\right\} \geq A^{2}\alpha _{n}^{2}\frac{n}{2\pi e}\left(
\beta -f(0)\frac{1}{2}\mathrm{vol}_{n}\left( A\alpha _{n}\sqrt{\frac{n}{2\pi
e}}B_{2}^{n}\right) \right) 
\]%
where $\alpha _{n}$ is such that $\mathrm{vol}_{n}\left( \alpha _{n}\sqrt{%
\frac{n}{2\pi e}}\right) =1$ and $\alpha _{n}\rightarrow 1$ as $n\rightarrow
\infty $ (and $B_{2}^{n}=\left\{ x:\left\vert x\right\vert \leq 1\right\} $%
). Optimizing in $A$ yields%
\begin{equation}
f(0)\geq Cn^{-3/2}\left( e\sqrt{2\pi }\right) ^{-n}
\label{lower bound at zero}
\end{equation}%
In the symmetric case one gets the optimal base $\sqrt{2\pi e}$. The
estimate $f(0)\geq 2^{-7n}$ can be found, for example, in \cite[Theorem 5.14]%
{LoVe}.

\section{\label{appendix thin shell implies CLT}Appendix: Thin shell implies
CLT}

For completeness we collect and prove various known results and tailor them
to our specific use. We refer the reader to \cite[Theorems 1.1 and 1.2, Eq.
(1.7) Proposition 3.1]{Bo} and \cite{ABP}\ for a more extensive discussion.
Our proof of Proposition 3.1 in \cite{Bo} on the Lipschitz constant of $%
\theta \mapsto M\left( \theta ,t\right) $ is slightly simplified.

\begin{theorem}
\label{thin shell imp CLT}Let $\varepsilon >0$. Let $\mu $ be a probability
measure on $\mathbb{R}^{k}$ with center of mass $0$, identity covariance,
and log-concave density $f=d\mu /dx$. If $\mu $ has the following thin shell
property:%
\[
\mu \left\{ x\in \mathbb{R}^{k}:\left\vert \frac{\left\vert x\right\vert }{%
\sqrt{k}}-1\right\vert >\varepsilon \right\} <\varepsilon 
\]%
then there exists $\Theta \subset S^{k-1}$ with $\sigma _{n-1}\left( \Theta
\right) \geq 1-C\sqrt{k}\exp \left( -ck\varepsilon ^{2}\right) $ such that
for all $\theta \in \Theta $,%
\[
\sup_{t\in \mathbb{R}}\left\vert \Phi \left( t\right) -\mu \left\{ x\in 
\mathbb{R}^{k}:\left\langle x,\theta \right\rangle \leq t\right\}
\right\vert \leq C\varepsilon 
\]
\end{theorem}

\begin{proof}
Write $M\left( \theta ,t\right) =\mu \left\{ x\in \mathbb{R}%
^{k}:\left\langle x,\theta \right\rangle \leq t\right\} $. For any $\theta
_{1},\theta _{2}\in S^{k-1}$ that are sufficiently close, say $\left\vert
\theta _{1}-\theta _{2}\right\vert <1/10$,%
\[
\left\vert M\left( \theta _{1},t\right) -M\left( \theta _{2},t\right)
\right\vert =\mu \left( M\left( \theta _{1},t\right) \Delta M\left( \theta
_{2},t\right) \right) 
\]%
where $A\Delta B=\left( A\backslash B\right) \cup \left( B\backslash
A\right) $ denotes the symmetric difference of $A$ and $B$. By projecting
onto $span\left\{ \theta _{1},\theta _{2}\right\} $ and identifying $%
span\left\{ \theta _{1},\theta _{2}\right\} $ with $\mathbb{R}^{2}$, we
conclude that%
\[
\left\vert M\left( \theta _{1},t\right) -M\left( \theta _{2},t\right)
\right\vert =\int_{-\infty }^{t}\int_{\left( 1-x\cos \beta \right) /\sin
\beta }^{\infty }q(x)dydx+\int_{t}^{\infty }\int_{-\infty }^{\left( 1-x\cos
\beta \right) /\sin \beta }q(x)dydx 
\]%
where $q$ is the density of the measure projection of $\mu $ into $E$
(identified with $\mathbb{R}^{2}$), see (\ref{proj def}), and $\cos \beta
=\left\langle \theta _{1},\theta _{2}\right\rangle $. By the Pr\'{e}%
kopa-Leindler inequality $q$ is log-concave, and defines a probability
measure with mean $0$ and identity covariance. It is an elementary fact that
for such a function, $q(x,y)\leq C\exp \left( -cx_{1}-cx_{2}\right) $ with
universal constants $C,c>0$. By a change of variables (through translation),%
\[
\left\vert M\left( \theta _{1},t\right) -M\left( \theta _{2},t\right)
\right\vert \leq 2C\int_{-\infty }^{0}\int_{t}^{-y\tan \beta }\exp \left(
-c^{\prime }x-c^{\prime }y\right) dxdy\leq Ce^{-c\left\vert t\right\vert
}\left\vert \theta _{1}-\theta _{2}\right\vert 
\]%
This implies that $M\left( \theta ,t\right) $ is $Ce^{-c\left\vert
t\right\vert }$-Lipschitz in $\theta $. Now let $\theta \in S^{k-1}$ be
chosen randomly, uniformly distributed on $S^{k-1}$ and let $F(t)=\mathbb{E}%
M\left( \theta ,t\right) $. By concentration on $S^{n-1}$ (see e.g. \cite%
{Kl2}) and the union bound, with probability at least $1-C\varepsilon
^{-1}\exp \left( -cn\varepsilon ^{2}\right) =1-C\sqrt{k}\left( k\varepsilon
^{2}\right) ^{-1/2}\exp \left( -ck\varepsilon ^{2}\right) $, the following
event occurs: for all $1\leq j\leq m$, $\left\vert M\left( \theta
,t_{j}\right) -F(t_{j})\right\vert <\varepsilon $, where $m=\left\lfloor
\varepsilon ^{-1}\right\rfloor $ and $t_{j}=F^{-1}\left( j/m\right) $. Using
monotonicity in $t$, we conclude that (with high probability) $\left\vert
M\left( \theta ,t\right) -F(t)\right\vert <C\varepsilon $ for all $t\in 
\mathbb{R}$. We now compare $F$ to $\Phi $. Let $\Phi _{k}\left( t\right) =%
\mathbb{P}\left\{ \sqrt{k}\theta _{1}\leq t\right\} $, where $\theta $ is
still uniform on $S^{k-1}$. Let $X$ be a random vector in $\mathbb{R}^{k}$
with distribution $\mu $ and independent of $\theta $. The vector $%
Y=\left\langle \theta ,k^{1/2}\left\vert X\right\vert ^{-1}X\right\rangle $
is independent of $k^{-1/2}\left\vert X\right\vert $ and has the same
distribution as $\theta _{1}$. Using Fubini's theorem and independence, and
assuming $t>0$,%
\begin{eqnarray*}
F(t) &=&\mathbb{P}\left\{ \left\langle \theta ,X\right\rangle \leq t\right\}
=\mathbb{P}\left\{ \frac{\left\vert X\right\vert }{\sqrt{k}}\left\langle
\theta ,\frac{\sqrt{k}X}{\left\vert X\right\vert }\right\rangle \leq
t\right\} =\mathbb{P}\left\{ Y\leq \frac{t\sqrt{k}}{\left\vert X\right\vert }%
\right\} \\
&=&\mathbb{P}\left\{ \left\vert \frac{\left\vert X\right\vert }{\sqrt{k}}%
-1\right\vert <\varepsilon \right\} \mathbb{P}\left\{ Y\leq \frac{t\sqrt{k}}{%
\left\vert X\right\vert }:\left\vert \frac{\left\vert X\right\vert }{\sqrt{k}%
}-1\right\vert <\varepsilon \right\} \\
&&+\mathbb{P}\left\{ \left\vert \frac{\left\vert X\right\vert }{\sqrt{k}}%
-1\right\vert >\varepsilon \right\} \mathbb{P}\left\{ Y\leq \frac{t\sqrt{k}}{%
\left\vert X\right\vert }:\left\vert \frac{\left\vert X\right\vert }{\sqrt{k}%
}-1\right\vert >\varepsilon \right\} \\
&\leq &1\cdot \Phi _{k}\left( \frac{t\sqrt{k}}{\left( 1-\varepsilon \right) 
\sqrt{k}}\right) +\varepsilon \cdot 1
\end{eqnarray*}%
A similar lower bound holds. For any $\delta ,x>0$,%
\begin{equation}
\Phi \left( \left( 1+\delta \right) x\right) -\Phi \left( x\right) \leq \Phi
^{\prime }\left( x\right) \delta x\leq C\delta  \label{log lipschitz}
\end{equation}%
It follows from rotational invariance of the standard normal distribution
and uniqueness of Haar measure that if $Z$ is a standard normal vector in $%
\mathbb{R}^{k}$ then $\sqrt{k}\left\vert Z\right\vert ^{-1}Z$ is uniformly
distributed on $\sqrt{k}S^{k-1}$. Simulating $\theta =\sqrt{k}\left\vert
Z\right\vert ^{-1}Z$,%
\[
\Phi _{k}(t)-\Phi _{k}(-t)=\mathbb{P}\left\{ \left\vert Z_{1}\right\vert
\leq tk^{-1/2}\left\vert Z\right\vert \right\} =\mathbb{P}\left\{ \left\vert
Z_{1}\right\vert \leq t\left( 1-\frac{t^{2}}{k}\right) ^{-1/2}\left( \frac{1%
}{k}\sum_{i=2}^{k}Z_{i}^{2}\right) \right\} 
\]%
which (after a bit of fiddling using (\ref{log lipschitz}) and Gaussian
concentration of $\left\vert Z\right\vert $ about $k^{1/2}$) implies the
well known estimate $\left\vert \Phi (t)-\Phi _{k}(t)\right\vert \leq
ck^{-1/2}$ for all $t\in \mathbb{R}$ (this can also be seen by considering
the density $\Phi _{k}^{\prime }$, similar details in \cite[Section 3]{Fr14}%
). Putting all this together,%
\[
F(t)\leq \Phi _{k}\left( \frac{t}{\left( 1-\varepsilon \right) }\right)
+\varepsilon \leq \Phi \left( \frac{t}{\left( 1-\varepsilon \right) }\right)
+\frac{C}{\sqrt{k}}+\varepsilon \leq \Phi \left( t\right) +C\varepsilon +%
\frac{C}{\sqrt{k}} 
\]%
with a similar lower bound. Similarly, this also holds for $t<0$.
\end{proof}

\end{document}